\documentclass[11pt]{amsart}

\usepackage{amssymb}
\usepackage{bm}
\usepackage{graphicx}
\usepackage[centertags]{amsmath}
\usepackage{amsfonts}
\usepackage{amsthm}
\usepackage{mathrsfs}

\theoremstyle{definition}
\newtheorem{thm}{Theorem}[section]
\newtheorem{dfn}[thm]{Definition}
\newtheorem{lem}[thm]{Lemma}
\newtheorem{prp}[thm]{Proposition}
\newtheorem{cor}[thm]{Corollary}

\newtheorem{thm*}{Theorem}

\newtheorem*{rmks}{Remarks}

\newtheorem*{rmk}{Remark}

\newtheorem*{clm}{Claim}

\newtheorem*{ntt}{Notation}

\newcommand{\inn}{\in\mathbb{N}}
\newcommand{\sch}{\mathbb{S}}

\newcommand{\e}{\varepsilon}
\newcommand{\de}{\delta}
\newcommand{\qpm}{_{q,p}^m}
\newcommand{\qpmk}{_{q,p}^{m_k}}
\newcommand{\blq}{B_{\ell_q}}
\newcommand{\blp}{B_{\ell_p}}
\newcommand{\SB}[2]{_{{#1},{#2}}}

\DeclareMathOperator{\supp}{supp}

\def\Tiny{\fontsize{4pt}{4pt}\selectfont}

\begin{document}

\title{Examples of $k$-iterated spreading models}

\author[S.A. Argyros, P. Motakis]{Spiros A. Argyros, Pavlos
Motakis}
\address{National Technical University of Athens, Faculty of Applied Sciences,
Department of Mathematics, Zografou Campus, 157 80, Athens,
Greece} \email{sargyros@math.ntua.gr, pmotakis@central.ntua.gr}

\maketitle


\footnotetext[2]{}

\begin{abstract}
It is shown that for every $k\inn$ and every spreading sequence
$\{e_n\}_{n\inn}$ that generates a uniformly convex Banach space
$E$, there exists a uniformly convex Banach space $X_{k+1}$
admitting $\{e_n\}_{n\inn}$ as a $k+1$-iterated spreading model,
but not as a $k$-iterated one.
\end{abstract}

\section*{Introduction}

The aim of the present note is to continue some research
initialized by B. Beauzamy and B. Maurey in \cite{m}. Before we
state our result, we need to recall the definition of $k$-iterated
spreading models. As is well known, spreading models are a central
concept in Banach space theory, invented by A. Brunel and L.
Sucheston in \cite{o}. For $k\geqslant 2$, the $k$-iterated
spreading models of a Banach space $X$, are inductively defined as
the spreading models of the spaces generated by the $k-1$-iterated
spreading models of $X$, where by 1-iterated, we understand the
usual spreading models.

H.P. Rosenthal asked whether the $k$-iterated, $k\geqslant 2$,
spreading models of any Banach spaces coincide with the
1-iterated. Beauzamy and Maurey answered that question, by showing
that the 2-iterated are, in general, different from the
1-iterated. More precisely they showed that there exists a Banach
space $X$, generating a spreading model, isomorphically containing
$\ell_1$ and $\ell_1$ is not a spreading model of the space $X$. A
more striking result in the same direction is given in \cite{n},
where it is shown that there exists a Banach space $X$, such that
every non trivial spreading model of $X$ isomorphically contains
$\ell_1$ and $\ell_1$ is not a spreading model of the space $X$.

In the present paper we separate the $k$-iterated and the
$k+1$-iterated spreading models, for every $k\inn$. More
precisely, the following is proved.

\begin{thm*} Let $\{e_n\}_{n\inn}$ be a spreading\footnote[3]{A sequence $\{e_n\}_{n\inn}$ in a seminormed space
$(E,\|\cdot\|_*)$ is called spreading if for every $n\inn$,
$k_1<\ldots<k_n\inn$ and $a_1,\ldots,a_n\in\mathbb{R},$ we have
that $\|\sum_{j=1}^na_j e_j\|_*=\|\sum_{j=1}^n a_j e_{k_j}\|_*$}
sequence generating a uniformly convex Banach space $E$. Then
there exists a sequence $\{X_k\}_{k\inn}$ of uniformly convex
Banach spaces, each one with a symmetric basis, such that for
every $k\inn$, the space $X_k$ admits a $k$-iterated spreading
model $\{\tilde{e}_n\}_{n\inn}$ equivalent to $\{e_n\}_{n\inn}$
and for every $i<k$, $E$ is not isomorphic to a subspace of the
space generated by any $i$-iterated spreading model of
$X_k$.\label{maintheorem}
\end{thm*}

Denoting by $\mathcal{SM}_k^\text{it}(X)$ the class of all
$k$-iterated spreading models of a Banach space $X$, it is an easy
observation that these classes define an increasing family, with
respect to $k$. The above mentioned family $\{X_k\}_{k\inn}$
satisfies the additional property that for every $k\inn$, the
family $\big\{\mathcal{SM}_i^\text{it}(X_k)\big\}_{i=1}^k$ is
strictly increasing.

It is worth pointing out, that the $k$-iterated spreading models
of a Banach space $X$, for $k\geqslant 2$, are not easily
visualized from the structure of the space $X$ and this is an
obstacle for studying the structure of the space generated by
them. The key property of the aforementioned sequence
$\{X_k\}_{k\inn}$, is that the space generated by a spreading
model of any $X_k, k\geqslant 2$, is either isomorphic to a
subspace of $X_k$, or to a subspace of $X_{k-1}$.

The definition of the sequence $\{X_k\}_{k\inn}$, uses some
features from the Beauzamy-Maurey construction and also the
classical result, that every space with an unconditional basis,
embeds into a space with a symmetric basis \cite{a,l,s}. In
particular, W.J. Davis' approach \cite{a}, based on the Davis,
Figiel, Johnson, Pelczynski interpolation method \cite{b}, is the
one which is the more convenient for our needs. Of independent
interest is also Proposition \ref{16}, characterizing the
structure of the spreading models of spaces with a 1-symmetric
basis.

\section{Preliminaries}

Our notation concerning Banach space theory, shall follow the
standard one from \cite{d}.

\begin{dfn}Let $(X,\|\cdot\|)$ be a Banach space and $(E,\|\cdot\|_*)$ a seminormed
space. Let $\{x_n\}_{n\inn}$ be a bounded sequence in $X$ and
$\{e_n\}_{n\inn}$ a sequence in $E$. We say that $\{x_n\}_{n\inn}$
generates $\{e_n\}_{n\inn}$ as a spreading model, if there exists
a sequence of positive reals $\{\de_n\}_{n\inn}$ with
$\de_n\searrow 0$, such that for every $n\inn, n\leqslant k_1
<\ldots<k_n$ and every choice $\{a_i\}_{i=1}^n\subset[-1,1]$ the
following holds:

\begin{equation*}
\Big|\big\|\sum_{i=1}^na_ix_{k_i}\big\| -
\big\|\sum_{i=1}^na_ie_i\big\|_*\Big| < \de_n
\end{equation*}

\end{dfn}

We also say that the Banach space $X$ admits $\{e_n\}_{n\inn}$ as
a spreading model, or $\{e_n\}_{n\inn}$ is a spreading model of
$X$, if there exists a sequence in $X$ which generates
$\{e_n\}_{n\inn}$ as a spreading model.

Brunel and Sucheston proved that every bounded sequence in a
Banach space, has a subsequence which generates a spreading model.
The main property of spreading models is that they are spreading
sequences, i.e. for every $n\inn, k_1 <\ldots<k_n$ and every
choice $\{a_i\}_{i=1}^n\subset\mathbb{R}$ we have
$\|\sum_{i=1}^na_ie_i\|_* = \|\sum_{i=1}^na_ie_{k_i}\|_*$.

We classify spreading sequences into four categories, with respect
to their norm properties. These are the trivial, the
unconditional, the singular and the non unconditional Schauder
basic spreading sequences.

A spreading sequence $\{e_n\}_{n\inn}$ is called trivial, if the
seminorm on the space generated by the sequence is not actually a
norm. In this case, if $E$ is the vector space generated by
$\{e_n\}_{n\inn}$ and $\mathcal{N} = \{x\in E: \|x\|_* = 0\}$,
then {\small $E$}$/_\mathcal{N} $ has dimension 1. It is also
worth mentioning that a sequence in a Banach space $X$ generates a
trivial spreading model, if and only if it has a norm convergent
subsequence. For more details see \cite{g,f}. From now on, we will
only refer to non trivial spreading models.

Otherwise we distinguish the three previously mentioned cases. A
spreading sequence is called singular if it is not trivial and not
Schauder basic. The definition of the other two cases is the
obvious one.

\begin{ntt}
(1) Let $E_0, E$ be Banach spaces. We write $E_0\rightarrow E$, if
$E$ is generated by a spreading sequence, which is a spreading
model of some seminormalized sequence in $E_0$. Also, for $k\inn$,
the notation $E_0\stackrel{k}{\rightarrow} E$, means that
$E_0\rightarrow E_1\rightarrow\ldots\rightarrow E_{k-1}\rightarrow
E$, for some sequence of Banach spaces $E_1,\ldots,E_{k-1}$.

(2) Let $E_0, E$ be a Banach spaces, such that $E_0$ has a
Schauder basis. We write $E_0\underset{\text{bl}}{\rightarrow}E$,
if $E$ is generated by a spreading sequence, which is a spreading
model of some seminormalized block sequence of the basis of $E_0$.
Also for $k\inn$, the notation $E_0\substack{k\\
\longrightarrow\\ \text{bl}}E$, means that
$E\underset{\text{bl}}{\rightarrow}E_1\underset{\text{bl}}{\rightarrow}\ldots\underset{\text{bl}}{\rightarrow}E_{k-1}\underset{\text{bl}}{\rightarrow}E$,
for some sequence of Banach spaces $E_1,\ldots,E_{k-1}$ with
Schauder bases.
\end{ntt}

\begin{dfn} Let $E_0$ be a Banach space, $\{e_n\}_{n\inn}$ be a
spreading sequence in a seminormed space, $k\inn$.
$\{e_n\}_{n\inn}$ is said to be a $k$-iterated spreading model of
$E_0$, if there exists a Banach space $E$, such that
$E_0\stackrel{k-1}{\rightarrow}E$, and $\{e_n\}_{n\inn}$ is the
spreading model of some seminormalized sequence in $E$.
\end{dfn}

\begin{dfn} Let $E_0$ be a Banach space with a Schauder basis, $\{e_n\}_{n\inn}$ be a
spreading sequence in a seminormed space, $k\inn$.
$\{e_n\}_{n\inn}$ is said to be a block $k$-iterated spreading
model of $E$, if there exists a Banach space $E$ with a Schauder
basis, such that $E_0\substack{k-1\\
\longrightarrow\\ \text{bl}}E$, and $\{e_n\}_{n\inn}$ is the
spreading model of some seminormalized block sequence of the basis
of $E$.
\end{dfn}

The following definition is in accordance with the corresponding
one for Schauder bases.

\begin{dfn} Let $\{e_n\}_{n\inn}, \{\tilde{e}_n\}_{n\inn}$ be non
trivial spreading sequences which generate the Banach spaces $E$
and $\tilde{E}$ respectively. $\{e_n\}_{n\inn},
\{\tilde{e}_n\}_{n\inn}$ are said to be equivalent, if the linear
map $e_n\rightarrow \tilde{e}_n$ extends to an isomorphism between
$E$ and $\tilde{E}$.
\end{dfn}

It is easy to see that if $X$ and $Y$ are isomorphic Banach
spaces, then any non trivial spreading model admitted by $X$ is
equivalent to one admitted by $Y$ and vice versa. Therefore we
introduce the following definition.

\begin{dfn} Let $\{x_n\}_{n\inn}$ be a sequence in a Banach space
$X$ which generates a non trivial spreading model
$\{e_n\}_{n\inn}$. Let also $\{\tilde{e}_n\}_{n\inn}$ be a non
trivial spreading model equivalent to $\{e_n\}_{n\inn}$. Then we
say that $\{x_n\}_{n\inn}$ isomorphically generates
$\{\tilde{e}_n\}_{n\inn}$ as a spreading model.
\end{dfn}

\section{Interpolating Spaces with a Symmetric basis}

We begin by presenting some estimations, concerning a sequence of
$\|\cdot\|\qpm$ norms, next defined.

\begin{dfn}Let $1\leqslant q < p$. For $m\inn$, define
$\|\cdot\|\qpm$ on $\ell_p$ as follows:

\begin{equation*}
\|x\|\qpm = \inf\Big\{\lambda>0: \frac{x}{\lambda}\in m\blq +
\frac{1}{m}\blp\Big\}
\end{equation*} \label{01}

\end{dfn}

\begin{rmks} The following statements are true for all $m\inn$

\begin{enumerate}

\item[1.] $\frac{1}{m}\|\cdot\|_p \leqslant \|\cdot\|\qpm
\leqslant (m + \frac{1}{m})\|\cdot\|_p$, thus
$\|\cdot\|\qpm\sim\|\cdot\|_p$.

\item[2.] If $x\in\ell_q$, then $\|x\|\qpm \leqslant
\frac{\|x\|_q}{m+\frac{1}{m}}$.

\item[3.] $\|\cdot\|\qpm$ is a symmetric norm, i.e. if
$\{a_i\}_{i\inn}\in\ell_p$, then

\begin{equation*}
\|\sum_{i=1}^\infty a_ie_i\|\qpm = \|\sum_{i=1}^\infty
\e_ia_ie_{\pi(i)}\|\qpm
\end{equation*}

for any $\{\e_i\}_{i\inn}$ choice of signs and any permutation
$\pi$ of the naturals.

\end{enumerate}

\end{rmks}

\begin{lem} Let $1\leqslant q<p, \{x_n\}_{n\inn}\subset \ell_p,
\e>0, M\in[\mathbb{N}]^\infty$, such that $\|x_n\|_p>\e$ for all
$n\inn$ and $\lim_{n\to\infty}\|x_n\|_\infty = 0$.

Then $\sup\big\{\|x_n\|\qpm: n\in\mathbb{N},m\in M\big\} =
\infty$. \label{02}
\end{lem}

\begin{proof} Towards a contradiction, suppose that $\sup\big\{\|x_n\|\qpm: n\in\mathbb{N},m\in
M\big\} < C$. Then for all $n\inn, m\in M$, there exist
$0<\lambda_n^m<C, y_n^m\in\blq, z_n^m\in\blp$ such that

\begin{equation*}
x_n = \lambda_n^m\big(my_n^m + \frac{1}{m}z_n^m\big)
\end{equation*}

Choose $m_0\in M$ such that
$\lambda_n^{m_0}\frac{1}{m_0}\|z_n^{m_0}\|_p < \frac{\e}{2}$, for
all $n\inn$. Then
\begin{equation}
\lambda_n^{m_0}m_0\|y_n^{m_0}\|_p > \frac{\e}{2},\quad\text{for
all}\; n\inn \label{eq1}
\end{equation}

By the symmetricity of the norms, we may assume that if $x_n =
\sum_{i=1}^\infty a_ie_i$, then $a_i\geqslant 0$, for all $i\inn$.
Moreover, if $y_n^m = \sum_{i=1}^\infty b_ie_i, z_n^m =
\sum_{i=1}^\infty c_ie_i$, we may assume that $0\leqslant
\lambda_n^mmb_i, \lambda_n^m\frac{1}{m}c_i \leqslant a_i$, for all
$i\inn$.

Otherwise, with simple calculations one may find $y_n^{m\prime} =
\sum_{i=1}^\infty b_i^\prime e_i, z_n^{m\prime} =
\sum_{i=1}^\infty c_i^\prime e_i$, satisfying this condition ,
such that $x_n = \lambda_n^m\big(my_n^{m\prime} +
\frac{1}{m}z_n^{m\prime}\big)$ and $y_n^{m\prime}\in\blq,
z_n^{m\prime}\in\blp$.

This means that $\lambda_n^{m_0}m_0\|y_n^{m_0}\|_\infty \leqslant
\|x_n\|_\infty \rightarrow 0$, as $n\rightarrow\infty$. Since
$\lambda_n^{m_0}m_0\|y_n^{m_0}\|_q < Cm_0$, for all $n\inn$, by
using the H\"older inequality, it is easy to see that
$\lambda_n^{m_0}m_0\|y_n^{m_0}\|_p\rightarrow 0$, as
$n\rightarrow\infty$. This contradicts \eqref{eq1}, which
completes the proof.

\end{proof}

\begin{lem} Let $1\leqslant q <p, \{x_n\}_{n\inn}\subset \ell_p,
M\in[\mathbb{N}]^\infty$ such that
$\lim_{n\to\infty}\|x_n\|_\infty = 0$ and $\sup\big\{\|x_n\|\qpm:
n\in\mathbb{N},m\in M\big\} < \infty$.

Then for every $\e>0, m_0\in M$, there exists $n_0\inn$ such that
for all $n\geqslant n_0: \max\big\{\|x_n\|\qpm: m\in M, m\leqslant
m_0\big\} < \e$.\label{03}
\end{lem}

\begin{proof} Towards a contradiction, suppose that there exist $\e>0, m_0\in M$, such that for all $n\inn$, there exists $k_n\geqslant n$
with $\max\big\{\|x_{k_n}\|\qpm: m\in M, m\leqslant
m_0\big\}\geqslant\e$.

By passing to subsequence of $\{x_n\}_{n\inn}$, we can find
$m\leqslant m_0$, such that $\|x_n\|\qpm\geqslant\e$, for all
$n\inn$.

But $\|\cdot\|\qpm\sim\|\cdot\|_p$, hence
$\|x_n\|_p\geqslant\e^\prime$, for all $n\inn$. By virtue of Lemma
\ref{02}, this means that $\sup\big\{\|x_n\|\qpm:
n\in\mathbb{N},m\in M\big\} = \infty$. Since this cannot be the
case, the proof is complete.

\end{proof}

The following Theorem is due to W.J. Davis and is presented in
\cite{a}. See also \cite{d}.

\begin{thm} Let $X$ be a (reflexive, uniformly convex) Banach space with a
1-unconditional basis. Then there exists a (reflexive, uniformly
convex) Banach space $D$ with a 1-symmetric basis, such that $X$
\makebox[2.95pt][s]{$\hookrightarrow$}%
\makebox[9pt][s]{\raisebox{-0.18ex}{\Tiny$\perp$}} $D$.

Moreover $D$ is saturated with subspaces of $X$, i.e. if $Z$ is a
subspace of $D$, then there exists a further subspace of $Z$,
which is isomorphic to a subspace of $X$. \label{00}
\end{thm}

Given $X$ a Banach space with a 1-unconditional basis, $D$ is
defined to be the diagonal subspace of $\mathfrak{X} =
\Big(\sum_{k=1}^\infty\bigoplus\big(\ell_p,|\!|\!|\cdot|\!|\!|\qpmk\big)\Big)_X$,
where the norms $|\!|\!|\cdot|\!|\!|_{q,p}^m$ are defined on
$\ell_p$ and are of the form $|\!|\!|x|\!|\!|_{q,p}^m =
\inf\{(\|y\|^2_{\ell_q} + \|z\|^2_{\ell_p})^\frac{1}{2}: x = my +
\frac{1}{m}z, y\in\ell_q, z\in\ell_p\}$, for $1<q<p$. If we denote
by $\tilde{e}_n = \{e_n,e_n,\ldots\}$, where $\{e_n\}_{n\inn}$ is
the natural basis of $\ell_p$, then $\tilde{e}_n\in\mathfrak{X}$
and $\{\tilde{e}_n\}_{n\inn}$ is the 1-symmetric basis of $D$.

Observe that for every $m\inn, \|\cdot\|\qpm\leqslant
|\!|\!|\cdot|\!|\!|\qpm\leqslant\sqrt{2}\|\cdot\|\qpm$. It
immediately follows that $\mathfrak{X}$, is isomorphic to
$\mathfrak{X}^\prime$, where $\mathfrak{X}^\prime =
\Big(\sum_{k=1}^\infty\bigoplus\big(\ell_p,\|\cdot\|\qpmk\big)\Big)_X$.

As it is shown in \cite{d}, $X$ embeds into $D$ as a complemented
subspace and every subspace of $D$ contains a further subspace
isomorphic to a subspace of $X$. The latter is shown in \cite{e},
but a proof also follows from the above and the following.

\begin{lem} Let $Y$ be a block subspace of $D$. Then there exists
a further block subspace $Z$ of $Y$, such that
$\lim_{n\to\infty}\|z_n\|_\infty = 0$, where $\{z_n\}_{n\inn}$
denotes the normalized block basis of $Z$.\label{x1}
\end{lem}

\begin{proof}
Let $\{y_n\}_{n\inn}$ be a normalized block basis of $Y$. If, by
passing to a subsequence, $\|y_n\|_\infty\rightarrow 0$, then  we
are ok. Otherwise, again by passing to a subsequence, there is
$\e>0$, such that $\|y_n\|_\infty>\e$, for all $n\inn$.

Denote by $j$ the map $j:D\rightarrow\ell_p$, with
$j\left(\sum_{i=1}^\infty a_i \tilde{e}_i\right) =
\sum_{i=1}^\infty a_i e_i$. As it can be seen in Davis' proof, $j$
is continuous. It follows that for $I$ a finite subset of the
naturals, $k\inn$: $\|\sum_{i\in I}\frac{1}{k}y_i\|_Y \geqslant
\|j\|^{-1}\frac{\e}{k}|I|^\frac{1}{p}$ and of course $\|\sum_{i\in
I}\frac{1}{k}y_i\|_\infty \leqslant \frac{1}{k}$.

Thus we can inductively construct a seminormalized block sequence
$\{z_n\}_{n\inn}$ of $\{y_n\}_{n\inn}$ such that
$\|z_n\|_\infty\rightarrow 0$.

\end{proof}

\begin{prp} Let $\{y_n\}_{n\inn}$ be a normalized bounded sequence
in $D$, such that $\lim_{n\to\infty}\|y_n\|_\infty = 0$. Then
$\{y_n\}_{n\inn}$ has a subsequence which is equivalent to a block
sequence in $X$. \label{07}
\end{prp}

\begin{proof}
By using Lemma \ref{03} and a sliding hump argument, with respect
to the decomposition
$\left\{\left(\ell_p,\|\cdot\|\qpmk\right)\right\}_{k\inn}$ of
$\mathfrak{X}^\prime$, the result follows.
\end{proof}

Since every subspace of $D$ contains a further subspace which is
isomorphic to a block subspace of $D$, it follows that $D$ is
saturated with subspaces of $X$.

\section{Uniformly Convex Schreier-Baernstein Spaces}

We begin by presenting some key definitions and results from
\cite{e}.

\begin{dfn} Let $X$ be a Banach space with a 1-unconditional basis
$\{e_n\}_{n\inn}$. The norm on $X$ is said to be $p$-convex
($q$-concave), if $\|\sum_{i=1}^n(|a_i|^p +
|b_i|^p)^\frac{1}{p}e_i\| \leqslant \big(\|\sum_{i=1}^na_ie_i\|^p
+
\|\sum_{i=1}^nb_ie_i\|^p\big)^\frac{1}{p}\quad\Big(\|\sum_{i=1}^n(|a_i|^q
+ |b_i|^q)^\frac{1}{q}e_i\| \geqslant
\big(\|\sum_{i=1}^na_ie_i\|^q +
\|\sum_{i=1}^nb_ie_i\|^q\big)^\frac{1}{q}\Big)$.

The norm on $X$ is said to satisfy an upper $\ell_p$ estimate
(lower $\ell_q$ estimate) if $\|x + y\|\leqslant \big(\|x\|^p +
\|y\|^p\big)^\frac{1}{p}\quad\Big(\|x + y\|\geqslant \big(\|x\|^q
+ \|y\|^q\big)^\frac{1}{q}\Big)$, whenever $x$ and $y$ are
disjointly supported with respect to the basis $\{e_n\}_{n\inn}$.
\end{dfn}

It is immediate, that if the norm on $X$ is $p$-convex
($q$-concave), then it satisfies an upper $\ell_p$ estimate (lower
$\ell_q$ estimate).

The following two results are restatements of Remark 3.2 and Lemma
3.1 respectively, from \cite{e}.

\begin{thm} Let $X$ be a uniformly convex Banach space with a
1-unconditional basis. Then there exists an equivalent
1-unconditional norm on $X$, which is $p$-convex and $q$-concave
for some $1<p\leqslant q<\infty$. \label{08}
\end{thm}

\begin{rmk} Provided that the initial norm on $X$ is spreading or 1-symmetric,
by it's own definition, it follows that the same is true for the
equivalent norm of Theorem \ref{08}.
\end{rmk}

\begin{lem} Let $X$ be a Banach space with a 1-unconditional
basis. If the norm on $X$ is $p$-convex and satisfies a lower
$\ell_q$ estimate, for some $1<p\leqslant q<\infty$, then $X$ is
uniformly convex. \label{09}
\end{lem}

Let $X$ be a Banach space with a 1-unconditional basis
$\{e_n\}_{n\inn}$.

We denote by $\sch$ the Schreier family $\sch =
\{F\subset\mathbb{N} : \min F \geqslant |F|\}$. Let $1\leqslant
r<\infty$. Define the following norm on $c_{00}(\mathbb{N})$:

\begin{equation*}
\|x\|\SB{X}{r} =
\sup\Big\{\big(\sum_{j=1}^d\|F_jx\|_X^r\big)^\frac{1}{r}\Big\}
\end{equation*}
Where the supremum is taken over all finite sequences
$\{F_j\}_{j=1}^d\subset\sch$ which are pairwise disjoint. Define
the Schreier-Baernstein space $SB\SB{X}{r}$, to be the completion
of $c_{00}(\mathbb{N})$ with the aforementioned norm.

It can be easily seen that the usual basis of $c_{00}(\mathbb{N})$
forms a 1-unconditional basis of $SB\SB{X}{r}$.

\begin{prp}Let $X$ be a Banach space with a 1-unconditional
basis $\{e_n\}_{n\inn}$, with a $p$-convex norm $\|\cdot\|_X$, for
some $1<p\leqslant\infty$. Let $r\geqslant p$. Then the space
$SB\SB{X}{r}$ is uniformly convex. \label{10}
\end{prp}

\begin{proof}

We will show that the demands of Lemma \ref{09} are satisfied.
First we show that $\|\cdot\|\SB{X}{r}$ is $p$-convex. Let
$\{a_i\}_{i=1}^n, \{b_i\}_{i=1}^n\subset\mathbb{R}$. Then for some
$\{F_j\}_{j=1}^d\subset\sch$ and by the $p$-convexity of the norm
on $X$:

\begin{eqnarray*}
\Big\|\sum_{i=1}^n\big(|a_i|^p\! +\!
|b_i|^p\big)^\frac{1}{p}e_i\Big\|\SB{X}{r}\!\!\!\! &=&\! \Bigg(
\sum_{j=1}^d\Big\|\sum_{i\in F_j}\big(|a_i|^p +
|b_i|^p\big)^\frac{1}{p}e_i\Big\|_X^r\Bigg)^\frac{1}{r}\\
&\leqslant&\!\! \Bigg( \sum_{j=1}^d\Big(\big\|\sum_{i\in
F_j}a_ie_i\big\|_X^p + \big\|\sum_{i\in
F_j}b_ie_i\big\|_X^p\Big)^\frac{r}{p}\Bigg)^{\frac{p}{r}\frac{1}{p}}\\
&\leqslant&\!\!\Bigg(\Big(\sum_{j=1}^d\big\|\sum_{i\in
F_j}a_ie_i\|_X^r\Big)^\frac{p}{r}\!\! +\!
\Big(\sum_{j=1}^d\big\|\sum_{i\in
F_j}b_ie_i\|_X^r\Big)^\frac{p}{r}\Bigg)^\frac{1}{p}\\
&\leqslant&\!\!\big(\|\sum_{i=1}^na_ie_i\|^p\SB{X}{r} +
\|\sum_{i=1}^nb_ie_i\|^p\SB{X}{r}\big)^\frac{1}{p}
\end{eqnarray*}
Thus $\|\cdot\|\SB{X}{r}$ is $p$-convex. Moreover, if $x,y$ are
finitely disjointly supported, then for some $\{F_j\}_{j=1}^{d_1},
\{E_j\}_{j=1}^{d_2}$:

\begin{equation*}
\|x\|\SB{X}{r} =
\big(\sum_{j=1}^{d_1}\|F_jx\|_X^r\big)^\frac{1}{r}\quad\text{and}\quad\|y\|\SB{X}{r}
= \big(\sum_{j=1}^{d_2}\|E_jy\|_X^r\big)^\frac{1}{r}
\end{equation*}
We may clearly assume that $F_i\cap E_j = \varnothing$ for all
$i,j$. Then:

\begin{equation*}
\|x+y\|^r\SB{X}{r} \geqslant \sum_{j=1}^{d_1}\|F_jx\|_X^r +
\sum_{j=1}^{d_2}\|E_jy\|_X^r = \|x\|^r\SB{X}{r} + \|y\|^r\SB{X}{r}
\end{equation*}
Thus $\|\cdot\|\SB{X}{r}$ satisfies a lower $\ell_r$ estimate and
the space $SB\SB{X}{r}$ is uniformly convex.

\end{proof}

\begin{prp}Let $X$ be a Banach space with a
1-unconditional basis $\{e_n\}_{n\inn}$, with a norm
$\|\cdot\|_X$, which satisfies a lower $\ell_q$ estimate, for some
$1\leqslant q<\infty$. Let $r\geqslant q$. Let $E\in\sch,
\{a_i\}_{i\in E}\subset\mathbb{R}$. Then $\big\|\sum_{i\in
E}a_ie_i\big\|\SB{X}{r} = \big\|\sum_{i\in E}a_ie_i\big\|_X$.
\label{11}
\end{prp}

\begin{proof}
It is clear by the definition of the norm on $SB\SB{X}{r}$ that
$\big\|\sum_{i\in E}a_ie_i\big\|\SB{X}{r} \geqslant
\big\|\sum_{i\in E}a_ie_i\big\|_X$. Therefore it is sufficient to
show the inverse inequality. For some $\{F_j\}_{j=1}^d\subset\sch$
and by the lower $\ell_q$ estimate of the norm on $X$:

\begin{equation*}
\big\|\sum_{i\in E}a_ie_i\big\|\SB{X}{r} =
\Big(\sum_{j=1}^d\big\|\sum_{i\in
F_j}a_ie_i\big\|^r_X\Big)^\frac{1}{r} \leqslant
\Big(\sum_{j=1}^d\big\|\sum_{i\in
F_j}a_ie_i\big\|^q_X\Big)^\frac{1}{q} \leqslant \big\|\sum_{i\in
E}a_ie_i\big\|_X
\end{equation*}

\end{proof}

\begin{cor} Let $X$ be a Banach space with a
1-unconditional and spreading basis $\{e_n\}_{n\inn}$, with a norm
$\|\cdot\|_X$, which satisfies a lower $\ell_q$ estimate, for some
$1\leqslant q<\infty$. Let $r\geqslant q$. Then the basis of
$SB\SB{X}{r}$ generates the basis of $X$ as a spreading model.
\label{12}
\end{cor}

The proof is an immediate consequence of Proposition \ref{11} and
the spreading property of the basis of $X$.

\begin{prp} Let $X$ be a Banach space with a 1-unconditional
basis $\{e_n\}_{n\inn}$. Let $1\leqslant r<\infty$. Let
$\{x_n\}_{n\inn}$ be a normalized block sequence in $SB\SB{X}{r}$,
such that $\lim_{n\to\infty}\|x_n\|_\infty = 0$. Then, by passing
to an appropriate subsequence, $\{x_n\}_{n\inn}$ is equivalent to
the basis of $\ell_r$. \label{13}
\end{prp}

The proof is the same as for the classical Screier-Baernstein
space $SB\SB{\ell_1}{2}$, where such a sequence has a further
subsequence which is equivalent to the basis of $\ell_2$. It also
follows that the space $SB\SB{X}{r}$ is $\ell_r$ saturated. If the
norm on $X$ satisfies a lower $\ell_q$ estimate and $r>q$, then
the space $X$ cannot contain $\ell_r$. Thus in this case the
spaces $X$ and $SB\SB{X}{r}$ are totally incomparable.

\section{Spreading Models of Banach Spaces with a Symmetric Basis}
In this section we study  the structure of the spreading models in
Banach spaces with a 1-symmetric basis. We start with the
following, which is critical for our proofs. This result and the
next Proposition are traced back to \cite{m} and are closely
related to Lemma IV.2.A.3 from \cite{f}.

\begin{prp} Let $X$ be a Banach space with a 1-symmetric and boundedly complete basis. Let $\{x_n\}_{n\inn}$ be a normalized block sequence
in $X$ and assume that there is some $\e>0$, such that
$\|x_n\|_{\infty} > \e$, for all $n\inn$. Then by passing to an
appropriate subsequence, there exist $\{y_n\}_{n\inn},
\{z_n\}_{n\inn}$ block sequences in $X$ and $\{u_n\}_{n\inn}$ a
disjointly supported 1-symmetric sequence in $X$ with the
following properties:
\begin{enumerate}

\item[i.] $x_n = y_n + z_n$, for all $n\inn$ and $\supp
y_n\cap\supp z_m = \varnothing$, for all $n,m\inn$.

\item[ii.] $\lim_{n\to\infty}\|z_n\|_\infty = 0$.

\item[iii.] $\{y_n\}_{n\inn}$ isometrically generates
$\{u_n\}_{n\inn}$ as a spreading model.

\item[iv.] $\|u_n\|_\infty = \|u_1\|_\infty > 0$, for all $n\inn$.

\end{enumerate}\label{16}
\end{prp}

\begin{proof}
Since the basis of $X$ is boundedly complete and symmetric, for
every $\de>0$, there exists $m(\de)\inn$, such that for every
$x\in X$ with $\|x\| = 1, \#\{i: |x(i)| \geqslant \de\} \leqslant
m(\de)$. Otherwise the basis of $X$ would be equivalent to the
basis of $c_0$.

Since $X$ has a 1-symmetric basis, we may assume that
$x_n(i)\geqslant 0$, for all $n,i\inn$ and the non zero entries of
each $x_n$ are in decreasing order.

For each $x_n$, we set $G_n = \supp x_n$. Let $G_n = \{i_1,\ldots
i_{d_n}\}$, then we set:

\begin{equation*}
\tilde{x}_n(\ell) = \left\{
\begin{array}{lcl}
x_n(i_\ell) &\text{if}&\;\ell\leqslant d_n\\
0 &\text{otherwise}&
\end{array}\right.
\end{equation*}

By passing, if it is necessary, to a subsequence, we may assume
that $\lim_{n\to\infty}\tilde{x}_n(i) = \lambda_i$, for each
$i\inn$. Since the basis of $X$ is boundedly complete, we conclude
that $x = \sum_{i=1}^\infty\lambda_ie_i\in X$.

Choose $\{\de_k\}_{k\inn}, \{\e_k\}_{k\inn}$ sequences of positive
reals, such that $\de_k,\e_k\searrow 0$.

Inductively we choose $n_1, n_2,\ldots,n_k,\ldots$, such that if
$k_0\inn$, then for every $k\geqslant k_0$ the following hold:

\begin{enumerate}

\item[a.] $\#\{i : \tilde{x}_{n_k}(i) \geqslant\de_{k_0}\} =
m_{k_0}$

\item[b.] $\|P_{[1,m_{k_0}]}(\tilde{x}_{n_k} - x)\| < \e_k$

\end{enumerate}

Define $\{y_k\}_{k\inn}$ as follows:

\begin{equation*}
y_k(i) = \left\{
\begin{array}{lcl}
x_{n_k}(i) &\text{if}\; x_{n_k}(i)\geqslant\de_k\\
0 &\text{otherwise}&
\end{array}\right.
\end{equation*}
and $z_k = x_{n_k} - y_k$. Properties i. and ii. are obviously
satisfied. Also observe that if $k_0\inn$, then for every
$k\geqslant k_0,\; P_{[1,m_{k_0}]}(\tilde{x}_{n_k}) =
P_{[1,m_{k_0}]}(\tilde{y}_k)$, where $\tilde{y}_k$ is defined in
the obvious way.

For $\de > 0, y\in X$ define $R^\de y\in X$ as follows:

\begin{equation*}
R^\de y(i) = \left\{
\begin{array}{lcl}
y(i) &\text{if}\; |y(i)|<\de\\
0 &\text{otherwise}&
\end{array}\right.
\end{equation*}

\begin{clm} $\lim_{\de\to 0}\sup\{\|R^\de y_k\|:
k\inn\} = 0$
\end{clm}

Let $\e>0$. Choose $k_0\inn$, such that
$\|\sum_{i=m_{k_0}}^\infty\lambda_ie_i\| < \frac{\e}{2}$ (if the
sequence $\{m_k\}_{k\inn}$ is bounded, then the same is true for
the support of $x$) and $\e_k < \frac{\e}{2}$, for all $k\geqslant
k_0$. Then for every $\de<\de_{k_0}$, by the definition of the
$y_k$, property b. and the fact that for $k\geqslant k_0,\;
P_{[1,m_{k_0}]}(\tilde{x}_{n_k}) = P_{[1,m_{k_0}]}(\tilde{y}_k)$
for every $k\geqslant k_0$, it follows that $\sup\{\|R^\de y_k\|:
k\inn\} < \e$, thus we have proved the claim.

Choose $\{N_k\}_{k\inn}\subset [\mathbb{N}]^\infty$ a partition of
the natural into infinite sets and set $u_k =
\sum_{i=1}^\infty\lambda_ie_{N_k(i)}$. It is immediate that
$\{u_k\}_{k\inn}$ is 1-symmetric and $\|u_k\|_\infty =
\|u_1\|_\infty > 0$, for all $k\inn$. We will prove that any
spreading model generated by $\{y_k\}_{k\inn}$ is isometric to
$\{u_k\}_{k\inn}$ by showing that if $\ell\inn,
\{a_i\}_{i=1}^\ell\subset [-1,1],$ then $
\lim_{j\to\infty}\|\sum_{i=1}^\ell a_iy_{i+j}\| =
\|\sum_{i=1}^\ell a_iu_i\|$.

Let $\e>0$. There exists $n_0\inn$, such that

\begin{equation}
\big|\|\sum_{i=1}^\ell a_iP_{N_i(n)}u_i\| - \|\sum_{i=1}^\ell
a_iu_i\|\big| < \frac{\e}{3},\quad \text{for all}\; n\geqslant n_0
\label{eqx2}
\end{equation}

By the claim, there exists $k\inn$ such that

\begin{equation}
\|\sum_{i=1}^\ell a_iR^{\de_k}y_{i+j}\| < \frac{\e}{3},\quad
\text{for all}\; j\inn\; \label{eqx3}
\end{equation}
and such that $m_k \geqslant n_0$.

Choose $j_0\inn$ such that $\#\supp \big\{y_{i+j} -
R^{\de_k}y_{i+j}\big\} = m_k$ for all $j\geqslant j_0$.

By property b. and the symmetricity of the base of $X$, there
exists $j_1\geqslant j_0$, such that

\begin{equation}
\big|\|\sum_{i=1}^\ell a_i(y_{i+j} - R^{\de_k}y_{i+j})\| -
\|\sum_{i=1}^\ell a_iP_{N_i(M_k)}u_i\|\big| < \frac{\e}{3},\quad
\text{for all}\; j\geqslant j_1.\label{eqx4}
\end{equation}

By combining relations \eqref{eqx2}, \eqref{eqx3} and
\eqref{eqx4}, it follows that $\big|\|\sum_{i=1}^\ell a_iy_{i+j}\|
- \|\sum_{i=1}^\ell a_iu_i\|\big| < \e$, for all $j\geqslant j_1$,
thus we have proved the proposition.

\end{proof}

The following is the corresponding result of Proposition \ref{16}
in the setting of Schreier-Baernstein spaces.

\begin{prp} Let $X$ be a Banach space with a
1-symmetric basis $\{e_n\}_{n\inn}$, with a norm which satisfies a
lower $\ell_q$ estimate, for some $1\leqslant q<\infty$. Let
$r\geqslant q$. Let $\{x_n\}_{n\inn}$ be a normalized block
sequence in $SB\SB{X}{r}$ and assume that there is some $\e>0$,
such that $\|x_n\|_{\infty}
> \e$, for all $n\inn$. Then by passing to an appropriate
subsequence, there exist $\{y_n\}_{n\inn}, \{z_n\}_{n\inn}$ block
sequences in $SB\SB{X}{r}$ and $\{u_n\}_{n\inn}$ a disjointly
supported 1-symmetric sequence in $X$ with the following
properties:
\begin{enumerate}

\item[i.] $x_n = y_n + z_n$, for all $n\inn$ and $\supp
y_n\cap\supp z_m = \varnothing$, for all $n,m\inn$.

\item[ii.] $\lim_{n\to\infty}\|z_n\|_\infty = 0$.

\item[iii.] $\{y_n\}_{n\inn}$ isometrically generates
$\{u_n\}_{n\inn}$ as a spreading model.

\item[iv.] $\|u_n\|_\infty = \|u_1\|_\infty > 0$, for all $n\inn$.

\end{enumerate}\label{14}
\end{prp}

The proof of this Proposition is similar to the one of Proposition
\ref{16} and uses Proposition \ref{11}.

\begin{cor}Let $X$ be a  Banach space with a
1-symmetric basis $\{e_n\}_{n\inn}$, with a norm which satisfies a
lower $\ell_q$ estimate for some $1\leqslant q<\infty$. Let
$r\geqslant q$. Let $\{x_n\}_{n\inn}$ be a normalized block
sequence in $SB\SB{X}{r}$ and assume that there is some $\e>0$,
such that $\|x_n\|_{\infty}
> \e$, for all $n\inn$. Then by passing to an appropriate
subsequence, there exists $\{u_n\}_{n\inn}$ a disjointly supported
1-symmetric sequence in $X$ such that:

\begin{enumerate}

\item[i.] $\{x_n\}_{n\inn}$ isomorphically generates
$\{u_n\}_{n\inn}$ as a spreading model.

\item[ii.] $\|u_n\|_\infty = \|u_1\|_\infty > 0$, for all $n\inn$.

\end{enumerate}\label{15}
\end{cor}

\begin{proof}
Apply Proposition \ref{14}, take the decomposition $x_n = y_n +
z_n$ and $\{u_n\}_{n\inn}$ the spreading model of
$\{y_n\}_{n\inn}$. By passing to a further subsequence, by virtue
of Proposition \ref{13} $\{z_n\}_{n\inn}$ is equivalent to the
basis of $\ell_r$. Then for $\ell\inn,
\{a_i\}_{i=1}^\ell\subset[-1,1]$, by standard arguments, keeping
in mind that the norm on $X$ satisfies a lower $\ell_q$ estimate
and the decomposition's properties, one can see that

\begin{equation*}
\|\sum_{i=1}^\ell a_iu_i\|_X \leqslant
\lim_{m\to\infty}\|\sum_{i=1}^\ell a_ix_{i+m}\|\SB{X}{r}\leqslant
C\|\sum_{i=1}^\ell a_iu_i\|_X
\end{equation*}
for some positive constant $C$.

\end{proof}

\section{The Main Result}

We start by stating some general facts about spreading models
admitted by a super reflexive Banach space $X$. As is well known,
the class of super reflexive Banach spaces, coincides with the one
of uniformly convex Banach spaces \cite{t}.

Suppose that $E$ is a Banach space, such that
$X\stackrel{k}{\rightarrow} E$, for some $k\inn$. Any space which
is finitely representable in $E$, is also finitely representable
in $X$, therefore $E$ must be super reflexive.

Thus any non trivial $k$-iterated spreading model
$\{e_n\}_{n\inn}$ of $X$ is weakly convergent. It follows that
$\{e_n\}_{n\inn}$ must either be 1-unconditional and weakly null,
or singular (see \cite{p} and Corollary 3.12 of \cite{g} or
\cite{f}).

Also if $\{e_n\}_{n\inn}$ is singular, then it is weakly
convergent to some element $e$ in $E = [\{e_n\}_{n\inn}]$ and if
we set $e^\prime_n = e_n - e$, then $\{e_n^\prime\}_{n\inn}$ is
1-unconditional, spreading and $E^\prime =
[\{e^\prime_n\}_{n\inn}]$ is isomorphic to $E$ (see Proposition
3.13 of \cite{g} or see \cite{f}).

\begin{lem} Let $\{e_n\}_{n\inn}$ be a singular and spreading
sequence. Then there exists a spreading and weakly null sequence
$\{d_n\}_{n\inn}$ with the following properties:

\begin{enumerate}

\item[i.] For every Banach space $X$ which admits
$\{e_n\}_{n\inn}$ as a spreading model, $X$ admits
$\{d_n\}_{n\inn}$ as a spreading model.

\item[ii.] For every Banach space $X$ and every $\{x_n\}_{n\inn}$
sequence in $X$ such that $\{x_n\}_{n\inn}$ generates
$\{d_n\}_{n\inn}$ as a spreading model, there exists a sequence
$\{y_n\}_{n\inn}$ in $[\{x_n\}_{n\inn}]$, such that
$\{y_n\}_{n\inn}$ isomorphically generates $\{e_n\}_{n\inn}$ as a
spreading model.

\end{enumerate}\label{lemmax1}
\end{lem}

\begin{proof}
Set $d_n = e_n^\prime$ as previously defined. Let $X$ be a Banach
space and $\{x_n\}_{n\inn}$ a sequence in $X$, which generates
$\{e_n\}_{n\inn}$ as a spreading model. Since $\{e_n\}_{n\inn}$ is
singular, $\{x_n\}_{n\inn}$ cannot contain a Schauder basic
subsequence. Thus it is neither equivalent to the basis of
$\ell_1$, nor nontrivial weak-Cauchy (see proof of Proposition 2.2
in \cite{r}). By Rosenthal's $\ell_1$ theorem (see \cite{q}), this
means that it is weakly convergent to some element $x\in X$. By
Lemma 5.8 of \cite{g}, if we set $x_n^\prime = x_n - x$, then
$\{x_n^\prime\}_{n\inn}$ generates $\{e_n\}_{n\inn}$ as a
spreading model. See also \cite{f}.

For ii. suppose that $\{x_n\}_{n\inn}$ is a sequence in $X$, that
generates $\{d_n\}_{n\inn}$ as a spreading model. By Rosenthal's
criterion for spreading sequences (Proposition 3.7, \cite{g}, see
also \cite{f}), $\{d_n\}_{n\inn}$ is Ces\`aro summable to zero.
Observe that this means that for every infinite subset $L$ of
$\mathbb{N}$, for any $\e>0$, one may find $F$ a finite subset of
$L$ and $\{\lambda_{i}\}_{i\in F}$ positive reals with $\sum_{i\in
F}\lambda_i = 1$, such that $\|\sum_{i\in F}\lambda_ix_i\| < \e$.
This means that $\{x_n\}_{n\inn}$ is weakly null. Otherwise there
would exist $\e>0, x^*\in S_{X^*}$ and $L$ an infinite subset of
$\mathbb{N}$, such that $x^*(x_n) > \e$, for all $n\in L$. This
contradicts our previous observation. Take $x$ a non zero element
in $[\{x_n\}_{n\inn}]$ and set $y_n = x_n + x$. By combining
Proposition 3.13 and Lemma 5.8 of \cite{g} (see also \cite{f}),
the result follows.
\end{proof}

\begin{lem}Let $X$ be a super reflexive Banach space with a
basis. Then every $k$-iterated spreading model of $X$ is
equivalent to a spreading sequence in the space generated by a
block $k$-iterated spreading model of $X$.\label{lemmax2}
\end{lem}

\begin{proof} We prove this Lemma by induction on $k$. Let $\{e_n\}_{n\inn}$ be a spreading model of $X$.
As we previously mentioned, it must be either 1-unconditional and
weakly null, or singular.

Suppose it is weakly null. If $\{x_n\}_{n\inn}$ is a sequence in
$X$ which generates $\{e_n\}_{n\inn}$ as a spreading model,
arguing as in the proof of Lemma \ref{lemmax1}, $\{x_n\}_{n\inn}$
is weakly null, thus has a subsequence equivalent to a block
sequence. By the way the block sequence is chosen, it is easy to
see that the block sequence actually isometrically generates
$\{e_n\}_{n\inn}$ as a spreading model.

If it is singular, then by i. of Lemma \ref{lemmax1}, $X$ admits
$\{d_n\}_{n\inn}$ as a spreading model, which is 1-unconditional
and weakly null. Apply the previous case. Then there is a block
sequence in $X$ that generates $\{d_n\}_{n\inn}$ as a spreading
model. Define $d_n^\prime = d_1 + d_{n+1}$. Then
$\{d_n^\prime\}_{n\inn}$ is a spreading sequence in
$[\{d_n\}_{n\inn}]$ and by using Proposition 3.13 of \cite{g} (see
also \cite{f}), one may prove that it is equivalent to
$\{e_n\}_{n\inn}$.

Observe that in either case, the space generated by the block
sequence in $X$, has a basis. This proves the statement for $k=1$.

Suppose that it is true for $k\inn$ and let $\{e_n\}_{n\inn}$ be a
$k+1$-iterated spreading model of $X$. Thus there exists a super
reflexive Banach space $E_k$, such that
$X\stackrel{k}{\rightarrow} E_k$, and $\{e_n\}_{n\inn}$ is a
spreading model of $E_k$. By the inductive assumption there exists
a supper reflexive Banach space $E_k^\prime$ with a basis, such
that
$X\substack{k\\
\longrightarrow\\ \text{bl}} E_k^\prime$ and $E_k\hookrightarrow
E_k^\prime$.

This means that $\{e_n\}_{n\inn}$ is equivalent to a spreading
model admitted by $E_k^\prime$. By applying the case $k=1$ for
$E_k^\prime$, the result follows.

\end{proof}

\subsection{A Sequence of Uniformly Convex Banach Spaces with a
1-symmetric Basis}

Given a uniformly convex Banach space $E$ with a 1-unconditional
and spreading basis $\{e_n\}_{n\inn}$, we shall inductively
construct a sequence of Banach spaces $\{X_k\}_{k\inn}$,
satisfying the following properties:

\begin{enumerate}

\item[i.] $X_k$ is uniformly convex and has a 1-symmetric basis
for all $k\inn$.

\item[ii.] $X_k$ is $\ell_{r_k}$ saturated, where
$\{r_k\}_{k\inn}$ is a strictly increasing sequence of positive
reals.

\item[iii.]$X_k$ and $E$ are totally incomparable for all $k\inn$.

\item[iv.]Any spreading model admitted by $X_{k+1}$, is either
equivalent to a spreading sequence in $X_k$, or generates a space
isomorphic to a subspace of $X_{k+1}$, for all $k\inn$.

\item[v.]The basis of $E$ is equivalent to a spreading model of
$X_1$ and the basis of $X_k$ is equivalent to a spreading model of
$X_{k+1}$, for all $k\inn$.

\end{enumerate}

Since $\{e_n\}_{n\inn}$ is weakly null and spreading, it is
1-unconditional. By virtue of Theorem \ref{08}, we may renorm $E$
such that the norm on $E$ is $p_0$-convex and $q_0$-concave, for
some $1<p_0\leqslant q_0 < \infty$. Choose $r_1>q_0$ and define
the Schreier-Baernstein space $SB\SB{E}{r_1}$. By Proposition
\ref{10}, $SB\SB{E}{r_1}$ is uniformly convex. Apply Theorem
\ref{00} and define $X_1$ to be the uniformly convex Banach space
with a 1-symmetric basis, which contains $SB\SB{E}{r_1}$.

Since $r_1>q_0$, by the fact that the norm on $E$ satisfies a
lower $\ell_{q_0}$ estimate, it cannot contain $\ell_r$, for
$r>q_0$. Also, as we have noted after Proposition \ref{13},
$SB\SB{E}{r_1}$ is $\ell_{r_1}$ saturated and by Theorem \ref{00},
the same is true for $X_1$. Thus the spaces $E$ and $X_1$ are
totally incomparable. By Corollary \ref{12}, since $X_1$ contains
$SB\SB{E}{r_1}$, $X_1$ isomorphically admits $\{e_n\}_{n\inn}$ as
a spreading model.

Assume that $\{X_i\}_{i=1}^k$ have been constructed. Using Theorem
\ref{08} once more, we may renorm $X_k$ such that the norm on
$X_k$ is $p_k$-convex and $q_k$-concave, for some $1<p_k\leqslant
q_k<\infty$. Choose $r_{k+1} > \max\{r_k,q_k\}$ and define the
Schreier-Baernstein space $SB\SB{X_k}{r_{k+1}}$, which is again
uniformly convex. Choose $1<s_{k+1}<t_{k+1}<p_k$ and construct the
uniformly convex space $X_{k+1}$ with a 1-symmetric basis, by
using the $\|\cdot\|_{s_{k+1},t_{k+1}}^m$ norms.

Again, $X_{k+1}$ is $\ell_{r_{k+1}}$ saturated, thus it is totally
incomparable with $E$ and $X_i$, for $i<k+1$. Also $X_{k+1}$
isomorphically admits the basis of $X_k$ as a spreading model. The
inductive construction is complete. Properties i., ii., iii. and
v. are obviously satisfied. We shall therefore only prove property
iv.

Let $k\geqslant 1$ and let $\{x_n\}_{n\inn}$ be a normalized block
sequence in $X_{k+1}$ that generates a spreading model. By passing
to a subsequence, one of the following is satisfied. Either
$\|x_n\|_\infty\rightarrow 0$, or there is an $\e>0$ such that
$\|x_n\|_\infty>\e$, for all $n\inn$.

If the first one is satisfied, then by Proposition \ref{07},
$\{x_n\}_{n\inn}$ has a subsequence isomorphic to a block sequence
in $SB\SB{X_k}{r_{k+1}}$. By Corollaries \ref{13} and \ref{15},
the spreading model generated by $\{x_n\}_{n\inn}$, is either
equivalent to a spreading sequence in $X_k$, or in
$\ell_{r_{k+1}}$. Since $X_{k+1}$ is $\ell_{r_{k+1}}$ saturated,
the spreading model is either equivalent to a spreading sequence
in $X_k$, or it is equivalent to a basic sequence in $X_{k+1}$.

If the second one is satisfied, then by Proposition \ref{16} pass
to a subsequence and take the decomposition $x_n = y_n + z_n$.
Then $y_n$ generates a disjointly supported symmetric spreading
model $\{u_n\}_{n\inn}$ in $X_{k+1}$, such that $\|u_n\|_\infty =
\|u_1\|_\infty>0$, for all $n\inn$.

Since the map $j: X_{k+1}\rightarrow\ell_{t+1}$ (the same map as
in the proof of Lemma \ref{x1}), is continuous

\begin{equation*}
\|\sum_{i=1}a_i\tilde{e}_i\|_{X_{k+1}} \geqslant
\|j\|^{-1}(\sum_{i=1}^\infty |a_i|^{t_{k+1}})^\frac{1}{t_{k+1}}
\end{equation*}

where $\{\tilde{e}_n\}_{n\inn}$ denotes the basis of $X_{k+1}$.

By this, $\|\sum_{i=1}^\infty a_iu_i\|_{X_{k+1}} \geqslant
C(\sum_{i=1}^\infty |a_i|^{t_{k+1}})^\frac{1}{t_{k+1}}$ for some
positive constant $C$.

Using Proposition \ref{07} again, by passing to a subsequence,
$\{z_n\}_{n\inn}$ is equivalent to a block sequence in
$SB\SB{X_k}{r_{k+1}}$. Thus the spreading model generated by
$\{z_n\}_{n\inn}$ is equivalent to the basis of $\ell_{r_{k+1}}$,
or equivalent to a disjointly supported sequence $\{v_n\}_{n\inn}$
in $X_k$. Remember that the norm on $X_k$ satisfies an upper
$\ell_{p_k}$ estimate and $t_{k+1}<p_k<r_{k+1}$.

Thus in either case, arguing as in the proof of Corollary
\ref{15}, the spreading model generated by $\{y_n\}_{n\inn}$,
absorbs the one generated by $\{z_n\}_{n\inn}$, thus
$\{x_n\}_{n\inn}$ isomorphically generates $\{u_n\}_{n\inn}$ as a
spreading model, which is contained in $X_{k+1}$. Therefore
property iv. holds for spreading models generated by block
sequences in $X_{k+1}$.

By virtue of Lemma \ref{lemmax2}, the same is true for any
spreading model admitted by $X_{k+1}$.

\begin{lem} The sequence $\{X_k\}_{k\inn}$ satisfies the following
additional property. For every $k\inn$ and for every $i<k$, if
$\{\tilde{e}_n\}_{n\inn}$ is an $i$-iterated spreading model of
$X_k$, then there exists $k-i\leqslant m\leqslant k$, such that
$[\{\tilde{e}_n\}_{n\inn}]$ is isomorphic to a subspace of $X_m$.
\label{18}
\end{lem}

\begin{proof}

Assume that the statement is true for $k\geqslant 2$.

Assume now, that $\{x_n\}_{n\inn}$ is an $i$-iterated spreading
model of $X_{k+1}$, for $i<k+1$. If $[\{x_n\}_{n\inn}]$ is
isomorphic to a subspace of $X_{k+1}$, then the statement is true
for $m= k+1$. If it is not, assume
$\big\{\{x_n^j\}_{n\inn}\big\}_{j=1}^i$, is the sequence of
spreading models leading to $\{x_n\}_{n\inn}$. Set $j_0 =
\min\big\{j: [\{x_n^j\}_{n\inn}]$ is not isomorphic to any
subspace of $X_{k+1}\big\}$. By property iv.,
$[\{x_n^{j_0}\}_{n\inn}]$ must be equivalent to a spreading
sequence in $X_k$. Thus $\{x_n\}_{n\inn}$ is equivalent to a
$(i-j_0)$-iterated spreading model of $X_k$. Since $i-j_0 < k$, by
the inductive assumption, there is $k+1-i\leqslant
k-i+j_0\leqslant m\leqslant k$, such that $[\{x_n\}_{n\inn}]$ is
isomorphic to a subspace of $X_m$.
\end{proof}

\begin{cor}The family
$\big\{\mathcal{SM}_i^\text{it}(X_k)\big\}_{i=1}^k$ is strictly
increasing, for all $k\inn$.\label{19}
\end{cor}

\begin{proof}
By Lemma \ref{18}, for $i<k$, any $i$-iterated spreading model
generates a Banach space isomorphic subspace of $X_m$, for
$k-i\leqslant m\leqslant k$.

Notice that $X_k$ isomorphically admits the basis of $X_{k-i}$ as
an $i$-iterated spreading model. If $i = k-1$, then $X_k$
isomorphically admits $\{e_n\}_{n\inn}$ as $i+1$-iterated
spreading model. If $i<k-1$, then $X_k$ isomorphically admits the
basis of $X_{k-i-1}$ as an  $i+1$-iterated spreading model.

In either case, since the spaces $\{X_k\}_{k\inn}, E$ are pairwise
totally incomparable, the result follows.
\end{proof}

\begin{proof}[Proof of Theorem \ref{maintheorem}]

If $\{e_n\}_{n\inn}$ is 1-unconditional, then the sequence
$\{X_k\}_{k\inn}$ is the desired one. This is an immediate
consequence of the properties i. to v. and Lemma \ref{18}

For the case in which $\{e_n\}_{n\inn}$ is not unconditional, it
must be singular. Apply Lemma \ref{lemmax1} and by keeping in
mind, that by the way the $d_n$ are chosen, $d_n\in E$ for all
$n\inn$, apply the previous case for $\{d_n\}_{n\inn}$.

By Lemma \ref{lemmax1}, $X_1$ isomorphically admits
$\{e_n\}_{n\inn}$ as a spreading model. Thus $X_k$ isomorphically
admits $\{e_n\}_{n\inn}$ as a $k$-iterated spreading model. Also
$E$ is isomorphic to $[\{d_n\}_{n\inn}]$ and the space generated
by an $i$-iterated spreading model of $X_k$, for $i<k$ is
isomorphic to a subspace of $X_m$, for $k-i\leqslant m\leqslant
k$. Since the spaces $X_m, E$ are totally incomparable, the result
follows.
\end{proof}

The methods used here make heavy use of the nice properties of
uniformly convex spaces. Therefore, although our result applies to
$\ell_p$ spaces, $1<p<\infty$, it remains unknown whether a
similar result can be stated for $c_0$ and $\ell_1$.

\end{document}